\documentclass[draft,a4paper]{amsart}

\usepackage{amssymb}
\usepackage{url}

\newtheorem{thm}{Theorem}
\newtheorem{prop}[thm]{Proposition}
\newtheorem{lem}[thm]{Lemma}
\newtheorem{cor}[thm]{Corollary}

\theoremstyle{remark}
\newtheorem{rem}[thm]{Remark}

\theoremstyle{definition}
\newtheorem{defn}[thm]{Definition}

\newcommand{\R}{\mathbb{R}}
\newcommand{\N}{\mathbb{N}}

\newcommand{\m}{\mathcal}

\DeclareMathOperator{\supp}{supp}

\begin{document}

\title{Jets of closed orbits of Ma\~n\'e generic Hamiltonian flows}

\author{C. M. Carballo}
\address{Universidade Federal de Minas Gerais,
Av. Ant\^onio Carlos 6627, 
31270-901, Belo Horizonte,
MG, Brasil.}
\email{carballo@mat.ufmg.br}

\author{J.  A. G. Miranda}
\address{Universidade Federal de Minas Gerais,
Av. Ant\^onio Carlos 6627, 
31270-901, Belo Horizonte,
MG, Brasil.}
\email{jan@mat.ufmg.br}


\maketitle

\begin{abstract}
We prove a perturbation theorem for the $k$-jets, $k\geq 2$, of the Poincar\'e map of 
a closed orbit of the Hamiltonian flow of a Tonelli Hamiltonian $H: T^*M\to \R$,
on a closed manifold $M$.
As a consequence we obtain Ma\~n\'e generic properties of Hamiltonian and 
Lagrangian flows.
\end{abstract}

\section{Introduction}

Let  $ M $ be  a closed Riemannian manifold 
and $T^*M $ its cotangent bundle endowed with the canonical symplectic structure 
$\omega= - d\Theta $, where $ \Theta $ is the Liouville $1$-form.
If $ \pi : T^* M \to M $ is the canonical projection, 
$\Theta$ is defined by
\[
\Theta_{(x,p)} (\xi)= p( d_{(x,p)} \pi (\xi) ),
\mbox{ for } \xi\in T_{(x,p)} T^* M.
\]
Let $ H :T^*M \to \R $ be  a  Hamiltonian of class $C^2$ that is  
{\it convex}, i.e., for each fiber $T_x^*M $, the restriction 
$H(x,p)$ has positive defined  Hessian, and 
{\it superlinear}, i.e., 
$\lim_{\| p\|\rightarrow \infty} \frac{H(x,p)}{\|p\|}=\infty $,
uniformly  in $ x\in M$.
A Hamiltonian like this is often called a {\em Tonelli} Hamiltonian.
The  Hamiltonian vector field $ X_H $  of $ H $ is defined by 
\[
\omega(X_H, \cdot)=dH(\cdot)
\]
and the Hamiltonian flow, $\phi_t^H: T^*M \to T^*M $, of $H$ 
is the flow corresponding to $X_H$.
Observe that this flow preserves the Hamiltonian function and the symplectic 
structure $\omega$.
The subsets $ H^{-1}(c) \subset T^*M $ are called 
energy levels of $H $.
Then the compactness of $ M $ and the superlinearity of $H$ 
imply that the nonempty energy levels are compact. 
Then the Hamiltonian flow is defined for all $ t \in \R$.    

The theory of Hamiltonian flows is closely related to that of the Lagrangian flows.
This is the reason to restrict to Tonelli Hamiltonians or Lagrangians.
Given a Tonelli, i.e., convex and superlinear,  Lagrangian 
$ L: TM \to \R $,  of class $ C^3 $, the Lagrangian flow of $L$ 
is conjugated to a Hamiltonian flow in $ T^*M$ by the 
Legendre transformation  $ \mathcal F_L : TM \rightarrow T^*M $,
defined by
\[ 
\mathcal{F}_L (x,v) = \left( x, \frac{\partial L}{\partial v} (x,v) \right).
\]
The corresponding convex Hamiltonian  $ H: T^*M \rightarrow \R $  
is
\[
H(x,p) = \max_{v \in T_xM} \{ p(v)- L(x,v)\}.
\]
Conversely, given a Hamiltonian $ H : T^* M \to \R $, 
the  corresponding  Lagrangian, defined by 
the  Legendre transformation of the Hamiltonian   
$\mathcal{F}_H : T^*M \rightarrow TM $, 
is
\[
L(x,v)= \max_{p \in T_x^* M } \{ p(v)-H(x,p)\}.
\] 
 
In this paper we  study  generic properties, in the sense of  Ma\~n\'e, 
of the jets of closed orbits of Tonelli Hamiltonian  flows.
Let  $ C^\infty (M) $ be   the space of smooth functions  
$ u: M \rightarrow \R$ endowed with the $C^\infty$ topology. 
Recall that a subset
$ \mathcal O  \subset C^\infty (M) $ is called  {\it residual} if it contains a countable
intersection of open and dense subsets.
We say that a property is {\em Ma\~n\'e generic}, if for each
Lagrangian $ L:TM\to \R$,  there exists  a residual subset   
$\m{O}  \subset  C^\infty (M) $,  
such that the property holds  for  all 
$L_u= L -u $, with  $ u \in  \m O$.
The corresponding definition in the Hamiltonian setting is obtained 
considering perturbations $H_u=H+u$. 

This genericity concept was introduced by R. Ma\~n\'e in \cite{Mane:1996a} to study
Aubry-Mather theory (see  \cite{Mather:1991}) for Lagrangian systems and, after this, 
many works were devoted to studying generic properties of Lagrangian and 
Hamiltonian systems in the sense  of Ma\~n\'e; see, for instance, 
\cite{Contreras:1999, Contreras:2002a, Massart:2003, Bernard:2008b}. 

Also in this context, E. Oliveira, in \cite{Oliveira:2008},  proved a version of the
Kupka-Smale theorem  for generic Hamiltonians on surfaces. 
In \cite{RR:2011}, L. Rifford and R. Ruggiero proved a perturbation theorem of 
the $1$-jet of the Poincar\'e map of a closed orbit for  Hamiltonians on  
a compact manifold of any dimension and, using this perturbation theorem and 
Oliveira's results, obtained a version of the Kupka-Smale theorem for Tonelli 
Hamiltonians and Lagrangians on closed  manifolds of any dimension.

\begin{thm}[Kupka--Smale]
\label{ks}
Let  $ M$ be a closed  manifold and $H: T^*M \to \R $ a $C^2$ Tonelli Hamiltonian.  
Then, for each $ c\in \R $, there exists a residual set 
$\m{KS} = \m{KS} (c) \subset C^\infty(M) $ such that every Hamiltonian  
$ H_u=H + u $, with  $ u \in  \m{KS}$, satisfies
\begin{enumerate}
\item
$H_u^{-1}(c)$ is a regular energy level,
\item
all closed orbits in $H_u^{-1}(c)$ are non-degenerate, and
\item
all heteroclinic intersections in $H_u^{-1}(c)$ are transverse.
\end{enumerate}
\end{thm}

The main  goal of the present paper is to complete this result
to include conditions on the higher order derivatives of the Poincar\'e maps of
closed orbits.  
It is  motivated by the fact that  
the dynamical behavior  near  elliptic  closed orbits  depend on 
the higher order derivatives  of the corresponding Poincar\'e map
(see Corollary \ref{twist-n}).

What we do here is analogous to what has been done for other 
classes of conservative systems.
Let us consider, for instance,  the geodesic flows, and recall
that a bumpy metric is a metric such that all closed geodesics are non-degenerate.
The bumpy metric theorem states that 
the subset of bumpy metrics is a residual subset of the space of smooth Riemannian metrics  
on $M$.
This  theorem is attributed to R. Abraham \cite{Abraham:1970}; 
see also D. V. Anosov \cite{Anosov:1982}, where a complete proof is given.
In \cite{Klingenberg:1972},
Klingenberg and Takens extend the bumpy metric theorem to include conditions on 
the $k$-jets of the Poincar\'e map over closed orbits for geodesic flows.
For the class of magnetic flows on surfaces, a complete study of  generic properties 
of closed orbits can be found in \cite{Miranda:2006a}.

Let us  recall some facts about  the jet space of symplectic diffeomorphisms.
We consider the space of smooth diffeomorphisms 
$f : \R^{m}\to \R^{m}$ such that $f(0) =0 $.
Given $ k \in \N $, we say that the diffeomorphisms $f$ and $g$ are 
{\em $k$-equivalent} if their Taylor polynomials of degree  $k$ at $0$ are equal.
The {\em $k$-jet} of a diffeomorphism $f$ at $0$, $j^k f (0)$, or $j^k f$ for short,
is the equivalence class of $f$.
If we consider only symplectic diffeomorphisms in $\R^{2n}$, 
the set of all the equivalence classes is the {\em space of symplectic $k$-jets}, 
that we denote $J^k_s(n) $.
Observe that $ J^k_s(n)$ is a vector space and that it is also a Lie group, 
with the product defined by
\[
j^k f \cdot j^k g = j^k ( f\circ g ).
\]

We say that a subset $Q \subset J^k_s(n)$ is {\em invariant} if
\begin{equation}
\label{inv-subset-J}
\sigma \cdot Q \cdot \sigma^{-1} = Q ,  \forall   \sigma \in J^k_s(n) .
\end{equation}

Note that  if $M$ has dimension $n+1$, 
$\theta$ is a nontrivial closed orbit of the Hamiltonian flow of $H:T^*M\to\R$
in  the energy level $H^{-1}(c)$, and
$\Sigma \subset  H^{-1}(c)$ is a local transverse section 
at the point $ \theta(0)$,
then $\Sigma$ also has a symplectic structure and the 
Poincar\'e map $P(\theta(0),\Sigma,L):\Sigma \rightarrow \Sigma $ is a symplectic
diffeomorphism.
Therefore, using Darboux coordinates, we can assume that
$j^k P(\theta(0),\Sigma,L))\in J^k_s(n)$.

Given  an invariant subset 
$Q \subset  J^k_s(n)$ and a  closed orbit $\theta$ of a Hamiltonian flow,  
it follows from (\ref{inv-subset-J}) that the property 
``the  $k$-jet of the  Poincar\'e map of $ \theta$ belongs to $Q $'' 
is  independent  of  the section $\Sigma$ and the coordinate system; 
hence, it is well defined.

Using the notation above, we state our perturbation theorem.
Recall that, given a  Tonelli Hamiltonian $H$, 
we denote the perturbed Hamiltonian $H_u=H+ u$.

\begin{thm}
\label{T-jet}
Let $M$ be a closed manifold of dimension $n+1$, $H: T^*M\to\R$,  
$k\geq 1$, and
$Q$ any open invariant subset of $J^k_s(n)$.
Suppose that $\theta(t)=(\gamma(t),p(t))$ is a nontrivial closed orbit 
of the Hamiltonian flow of $H$.
If the $k$-jet of the Poincar\'e map of $\theta$ is in $\overline{Q}$,
then there exists a smooth potential $u: M\to \R $, 
$C^r$-close to zero, $r\geq k+1$, such that
\begin{enumerate}
\item
$\theta(t)$ is also a closed orbit of the Hamiltonian flow of $H_u$  and
\item
the $k$-jet of the Poincar\'e map of $\theta$, 
as a closed orbit of the Hamiltonian flow of $H_u$, is in $Q$.
\end{enumerate}
\end{thm}

In the particular case $k=1$ the theorem was proved by L. Rifford and R. Ruggiero:
Proposition 4.2 and Lemma 4.3  in \cite{RR:2011}  show that 
an open set of  $ Sp(n)=J^1_s(n)$ is attained by a family of differentials of the 
Poincar\'e maps of the closed orbit $ \theta $ obtained by perturbing the initial Hamiltonian $ H $ by adding small potentials.
See also \cite[Theorem 4.5]{Oliveira:2008}.
Our contribution is the case of higher order jets, $k\geq 2$.

Combining the Kupka-Smale theorem stated above, Theorem \ref{ks}, and 
our  perturbation  theorem, we obtain the genericity result.

\begin{thm}
\label{prop-P-Q}
Let $ M $ be a closed manifold of dimension $n+1$.
For every open, dense, and invariant subset 
$Q\subset J^k_s(n) $ and every $c\in \R$, 
there is a residual subset $\mathcal{O}=\mathcal{O}(Q, c)$ of $C^\infty(M)$ such that, 
if $u\in \mathcal{O}$, 
then the $k$-jet of the  Poincar\'e map of any closed orbit of the Hamiltonian flow of
$H_u= H+ u$ in $H_u^{-1}(c)$ is in $Q $.
\end{thm}

One needs Theorem \ref{prop-P-Q} to study the dynamics of a Ma\~n\'e generic 
Hamiltonian flow that has a  non-hyperbolic  closed orbit in a regular energy level.
Recall that a closed orbit $\theta$ in $H^{-1}(c)$ is {\em $q$-elliptic}
if the derivative of its Poincar\'e map $P$ has exactly $2q$ eigenvalues 
of modulus $1$ which are non-real;
it is {\em quasi-elliptic} if it is $q$-elliptic for some $q>0$. 
If $\theta$ is a $q$-elliptic closed orbit, then the central manifold $W^c$ of the 
Poincar\'e map $P:\Sigma\to \Sigma$ of $\theta$ has dimension $2q$, 
$\omega|_{W^c}$ is non-degenerate, and $F=P|_{W^c}$ is a symplectic map on a
sufficiently small neighborhood of $\theta$.

Let 
$\lambda_1, \dots, \lambda_q, \overline{\lambda_1}, \dots, \overline{\lambda_q}$ 
be the eigenvalues of modulus $1$ of the derivative $d_\theta P$ of the Poincar\'e map of 
a $q$-elliptic orbit $\theta$.
We say that $\theta$ is {\em $4$-elementary} if
$\prod_{i=1}^q \lambda_i^{m_i}\neq 1$ 
whenever the integers $m_i$ satisfy $1\leq\sum_{i=1}^q |m_i|\leq 4$.

\begin{thm}[Birkhoff Normal Form]
\label{bnf}
Let $F: \R^{2q}\to \R^{2q}$ be a symplectic diffeomorphism  such that $0$ is a 
$q$-elliptic $4$-elementary fixed point for $F$.  

Then
there are symplectic coordinates $(x,y)= (x_1, \dots, x_q, y_1, \dots, y_q)$ 
in a neighborhood of $0$ such that
\[
F_k(x,y)=F_k(z)=e^{2\pi i\phi_k(z)} z_k+ R_k(z), \quad k=1, \dots, q,
\]
where
$z= x+ iy$, $\phi_k(z)= a_k+ \sum_{l=1}^q \beta_{kl} |z_l|^2$,
$\lambda_k= e^{2\pi i a_k}$,
and $R_k(z)$ has zero derivatives up to order $3$ at $0$. 
\end{thm}
For a proof, see \cite[p.~101]{Klingenberg:1978}.

Then, we say that a $q$-elliptic closed orbit $\theta$ is {\em weakly monotonous}
if the coefficients $\beta_{kl}$ of $F=P|_{W^c}$ satisfy
$\det[\beta_{kl}]\neq 0$.

Let $Q\subset J^3_s(n)$ be the set of $3$-jets of $C^\infty$ symplectic diffeomorphisms
$T$ of $\R^{2n}$ such that the origin is 
\begin{enumerate}
\item
a hyperbolic fixed point or
\item
a weakly monotonous $4$-elementary $q$-elliptic fixed point, for some $q>0$.
\end{enumerate}

Observe that $Q$ is open, dense and invariant; recall the definition of invariant, 
equation (\ref{inv-subset-J}).
Let $c\in\R$ define the energy level.
By applying Theorem \ref{prop-P-Q}, to $Q$ and $c$, we get a residual subset $\mathcal{O}$.
Now, if $u\in \mathcal{O}$ and  $\theta$ is a non-hyperbolic orbit of $H_u$ 
in the energy level $H_u^{-1}(c)$, 
then the $3$-jet of the Poincar\'e map of $\theta$ is in $Q$.
Hence, $\theta$ is a weakly monotonous $q$-elliptic orbit, for some $q>1$.
This proves the following result.

\begin{cor}
\label{twist-n}
Let $M$ be a closed manifold, $H: T^*M\to \R$ a Tonelli Hamiltonian and $c\in\R$.
There is a residual subset $\mathcal{O}= \mathcal{O}(c)$ of $C^\infty(M)$ such that,
for every $u\in\mathcal{O}$,
any closed orbit of $H_u$ in $H_u^{-1}(c)$ is either hyperbolic or weakly monotonous 
quasi-elliptic.
\end{cor}
 
Observe that what we called here weakly monotonous is a generalization of the twist condition
of surface diffeomorphisms.
This property has important dynamical consequences such as 
the existence of a nontrivial hyperbolic set 
and then, infinitely many periodic orbits and positive topological entropy.
We refer to \cite{Contreras:2010} for the details. 
In this paper G. Contreras treats the case of geodesic flows but 
the arguments apply to our context as well.

\section{The $k$-general position}
\label{k-general maps}

In this preliminary section, we introduce the idea of $ k$-general position through 
an open and dense set $ G_k $.
Later we will use this concept to perform the perturbation of the jets. 

We identify  $J^1_s(n) $ with the classic Lie group $Sp(n)$ 
of symplectic maps in $\R^{2n}$ which is also identified with the group of 
$2n\times 2n$ symplectic matrices, i.e., 
with the matrices that satisfy $\sigma^T  J  \sigma = J$,
where
\[
J = 
\begin{bmatrix}
 0 & I \\ 
-I & 0  
\end{bmatrix}
\] 
and $I=I_n$ is the $n\times n$ identity matrix.

For each $k \in \N $,  we denote by  
$\R_k[x, y]$  the set of all real homogeneous polynomials of degree $k$ 
in the $2n$ variables $x=(x_1,\dots, x_n)$, $y=(y_1,\dots, y_n)$. 
This is a real vector space of dimension $d=d(2n,k)= \binom{2n - 1 + k}{k}$.

We fix the polynomial 
$F(x_1,\dots,x_n,y_1,\dots,y_n) = x_1^k$ and we define
\[
G_k =\left\{ 
(\sigma_1,\dots,\sigma_d) \in Sp(n)^{d} : 
\left\{ F\circ\sigma_1, \dots , F\circ\sigma_d \right\} 
\mbox{ is a basis  of } \R_k[x, y] \right\}.
\]

\begin{prop}
\label{P-kgen}
For each  $ k \in \N $, the subset  $ G_k $ is open and dense in  $Sp(n)^d$.
\end{prop}

\begin{proof}
Let $F_1, \dots, F_d$ be a basis of the vector space $\R_k[x, y]$ and define
the matrix $A(\sigma_1, \dots, \sigma_d)=[a_{ij}]$ by
the equation $F\circ\sigma_j= \sum_i a_{ij} F_i$.
Then $ G_k $ is the  complement in $Sp(n)^d $ of the set 
\[
\left\{ 
(\sigma_1, \dots,\sigma_d) \in Sp(n)^d : 
\det A(\sigma_1, \dots, \sigma_d)= 0 \right\}.
\] 
This shows that $ G_k $ is the complement of an algebraic
subset.

To verify that $G_k$ is non-empty, let us consider the following example.
Define
\[
\sigma
= \begin{bmatrix}
A &   B  \\ 
0 & ( A^{-1} )^T 
\end{bmatrix}
\]
where 
\[
A
= \begin{bmatrix}  
 1  &|& a_2  \dots a_n\\
 --&|&------\\
 0  &|& \\
\vdots    &|&  I_{n-1} \\
0   &|&
\end{bmatrix}, \quad
B
=\begin{bmatrix}
a_{n+1} & | & a_{n+2} \dots a_{2n}\\
 --        & | & ------\\
b_2       & | & \\
\vdots  & | & I_{n-1}\\
b_{n}     & | &
\end{bmatrix},
\]
and $b_i= a_{n+i} - a_i $, for $ i= 2, \dots, n$.
It is easy to check that the matrix $A^{-1} B$ is symmetric; 
hence, $\sigma$ is symplectic.

Then the composition $F \circ \sigma$ defines the homogeneous polynomials
\[
( x_1 + a_2 x_2+\dots+a_{n} x_n+ a_{n+1} y_1+\dots+ a_{2n} y_n )^k.
\]
Let us show that there exist values $a_{ij} $ such that 
the corresponding set of polynomials contain a basis of 
$\R_k[x, y]$.  
To simplify the notation we write the argument using arbitrary variables 
$x_1,\dots, x_m$.
First, observe that
\[
( x_1 + a_2 x_2+ \cdots+ a_m x_m)^k=
\sum_{\alpha_1+\cdots +\alpha_m= k} 
a_2^{\alpha_2}\cdots a_m^{\alpha_m}\
c_{\alpha_1\cdots \alpha_m} x_1^{\alpha_1}\cdots x_m^{\alpha_m},
\]
where all the coefficients $c_{\alpha_1\cdots \alpha_m}$ are different from zero 
so that the monomials 
$c_{\alpha_1\cdots \alpha_m} x_1^{\alpha_1}\cdots x_m^{\alpha_m}$ 
are a basis of the space of homogeneous polynomials of degree $k$.
Then, we let $a_i= t^{p_i}$ and need to choose the powers $p_i$ to ensure that 
after this replacement we have different powers of $t$ in all the 
$a_2^{\alpha_2}\cdots a_m^{\alpha_m}$.
One can check that a possible choice is $p_i= \sum_{j=0}^{i-2} k^j$, $i=2, \dots, 
m$.
Now that the scalars are all powers of $t$, the result follows using the result about
Vandermonde matrices.
\end{proof}

\section{Proof of Theorem~\ref{T-jet}}
\label{k-jet-perturbation}

Let $ \theta=\theta(t)=(\gamma(t), p(t))$  be a  closed orbit 
of the Hamiltonian flow  of the initial Hamiltonian $H_0=H: T^* M \to \R$ and 
let $ T >0 $ be its minimal period.
Recall that, by the Kupka-Smale Theorem, we can assume  that $ \theta $ is 
a non-degenerate closed orbit and that the  energy level $H_0^{-1}(c) $ that contains $ \theta $ 
is  regular.

Since the number of self-intersection points is finite, we can choose  
a time $ a\in ( 0,T ] $, such that the arc $ \gamma( [0,a] ) $ does
not contain self-intersection points of the  curve $ \gamma $ and
a tubular neighborhood $ W \subset M $ of  $ \gamma( [0,a] ) $,
sufficiently small,  such that  $ W \cap \gamma([0,T]) = \gamma( [0,a] )$.

Then, we  choose local coordinates 
$(x_0,x): W \to \R \times \R^n $, $ x=( x_1,\dots, x_n)$, such that 
\[
x_0(\gamma(t))=t \ \ \mbox{ and }\ \ \  x(\gamma(t))= ( 0,\dots, 0),
\] 
for $t\in [0,a]$. 
Then, if $p\in T^*_{(x_0,x)}M$, we define $y_i$ by
$p=\sum y_i dx_i$ and we have a natural chart 
\[
(x_0, x, y_0, y)= (x_0, x_1,\dots, x_n, y_0, y_1, \dots, y_n)
\] 
of $\pi^{-1}(W) \subset T^*M$.

In these coordinates we have that the $2$-form $\omega$ that defines
the symplectic structure in $T^* M$ can be written
$\omega = \sum_{i=0}^n dx_i \wedge dy_i$
and that the Hamiltonian vector field $X_H$ of a Hamiltonian $H$ is
\begin{equation}
\label{X-loc}
X_{H}= \sum_{i=0}^n 
\left( \frac{\partial H}{\partial y_i}\frac{\partial}{\partial x_i } -
 \frac{\partial H}{\partial x_i} \frac{\partial}{\partial y_i } \right).
\end{equation}

Now, we  define our  perturbation space.

Given $ k \in \N$, 
let  $ \mathcal{F}^k =\mathcal{F}(W,\gamma,H_0,a,k)$  be  the  subset of smooth functions   
$ u: M\rightarrow \R  $  supported in $ W$  such that,  in the local coordinates $(x_0,x)$, 
have the form
\[ 
u(x_0, x_1, \dots, x_n) =  \delta(x_0)\beta(x_1, \dots, x_n),
\]
where $\delta:\R\to\R $ and $\beta: \R^n\to\R$ are smooth functions satisfying 
\begin{itemize}
\item
$\supp(\delta)\subset ( 0,a)$ 
\item
$\supp(\beta) \subset B_\epsilon(0)$, 
with  $\epsilon$ sufficiently small, and
\item
$j^{k+1} \beta(0)$ is homogeneous of degree $k+1$.
\end{itemize}
We will  consider perturbations $H_u= H+ u$ of the initial Hamiltonian $ H_0=H$, 
with  $ u \in \m{F}^k \subset C^\infty (M)$.

Then $ X_{H_u} = X_{H_0} + X_u $ and, by (\ref{X-loc}), we have
\begin{equation}
\label{X-u-loc}
X_u (x_0, x, y_0, y)
= -\delta^\prime(x_0)\beta(x)\frac{\partial}{\partial y_0}
    -\sum_{i=1}^n
      \delta(x_0)\frac{\partial\beta}{\partial x_i}(x)\frac{\partial}{\partial y_i}.
\end{equation}
Note that, if $ u \in \mathcal{F}^k$, then the  closed orbit $ \theta $ is  still 
a closed orbit of the Hamiltonian flow of $H_u $ 

Using the coordinates $ (x_0, x, y_0, y)$  in $\pi^{-1}(W) \subset T^*M$, 
we set $\Lambda(t)=\{ x_0=t\}$.
This defines a family of local hypersurfaces in $T^*M$ that are transverse to 
$\theta(t) $,  for $ t\in [0,a]$.  
It follows from the definition of the set  $\mathcal{F}^k$ that
the vector field $X_u$ satisfy
\begin{itemize}
\item
$j^{k-1} X_u (\theta(t))=0$, for all $t \in [0,a] $,
\item
neither $\theta(0)$ nor $\theta(a)$ is in $\supp (X_u) $, and
\item
$ X_u|_{\Lambda(t)} $ is tangent to $\Lambda(t) $, for all $t \in [0,a] $.
\end{itemize}
Also, for each $t \in [0,a]$, we denote
\[
R_{u,t}= \widehat{P}_{0,t}^{-1} \circ \widehat{P}_{u,t},
\]
where $\widehat{P}_{0,t}: \Lambda(0) \rightarrow \Lambda(t)$
and  $\widehat{P}_{u,t}: \Lambda(0) \rightarrow \Lambda(t) $ are
the Poincar\'e maps in an open neighborhood of  $\theta(0)\in \Lambda(0)$
with respect to $H_0$ and $H_u$ respectively.

The following proposition holds for abstract vector fields satisfying  
the three  conditions above.
For a proof, see \cite[section 2]{Klingenberg:1972}.

\begin{prop}
\label{O-G}
The $k$-jet of $R_{u,a}$ at $\theta(0) $
is equal to the $k$-jet  of the time-$a$ map of the flow of
the non autonomous vector field 
$ \widehat{P}_{0,t}^* ( \left. X_u\right|_{\Lambda(t)} ) $ at $\theta(0)$.
\end{prop}

We are interested in the $k$-jet of the map $R_{u,a}$  
restricted  to the energy level  $ H_0^{-1}(c)$. 
First, note that  the energy levels $ H_0^{-1}(c) $ and $ H_u^{-1}(c) $  are  
$k$-tangent   along  the closed orbit $ \theta $, for all  $ u \in  \m{F}^k$. 
 
For each $ u \in \m{F}^k$,  let $ \Sigma_u(t) \subset T^*M $ be the submanifold given by
$\Sigma_u(t)=\Lambda(t) \cap H_u^{-1}(c)$.
Then $\omega $ induces a symplectic structure on $ \Sigma(t) $ and the restriction  
$\widehat{P}_{0,t}|_{\Sigma_0(0)} : \Sigma_0(0)\to\Sigma_0(t)$ 
is a symplectic map for all $ t \in [0,a]$. 
Since neither $\theta(0)$ nor $\theta(a)$ is in $\supp (X_u) $, the restriction  
$\widehat{P}_{u,a}|_{\Sigma_0(0)} : \Sigma_0(0)\to\Sigma_0(a) $ 
is a symplectic map too.

Observe that 
$ \frac{\partial H_0}{\partial y_0} (\theta(t)) \equiv 1$.
Then we can parameterize   $\Sigma_0(t) $ in terms of the coordinates  
$(x,y)$, this is,  for each $ t \in [0,a]$ 
there is an open set $V_t \subset\R^{2n}$ and a  function 
$\alpha_t:V_t \to\R $, such that
\[
\Sigma_0(t)= \{ (t, x, \alpha_t(x,y), y)  \in\Lambda(t) : (x, y) \in V_t \}.
\] 
Since 
$T\Sigma_0(t)\subset \ker(dx_0)$, 
the symplectic structure induced by  $ \omega$ in $ \Sigma_0(t) $ 
is given by
$\omega|_{\Sigma_0(t)} = dx \wedge dy$.
For each $ u \in\mathcal F^k $ and $ t \in [0,a]$, we consider the Hamiltonian
function $ K_{u,t}:\Sigma_0(t) \rightarrow \R $ given by 
\[
K_{u,t} 
= u|_{\Sigma_0(t)}
= \delta(t)\beta(x_1, \dots, x_n)
\] 
and we denote by $Y_{u,t} $
its Hamiltonian vector field.

Then
\begin{equation}
\label{conta-jet}
 j^{k+1} K_{u,t} (\theta(t)) = \delta(t) x_1^{k+1}
\end{equation}
and this defines a  family,  parameterized by  $t$, of multiples of
the polynomial  
\[
F(x_1,\dots,x_n,y_1,\dots, y_n)= x_1^{k+1}.
\]

\begin{rem}
\label{Sigma-inv}
If $(t, x) \in\supp(u)$, the submanifold $ \Sigma_0(0) $ is not invariant by the map $R_{u,t}$.
But, 
the $k$-jet of the  component of  $X_u $ in the direction of $\frac{\partial}{\partial y_0} $ 
at $\theta(t) $ is zero, see  (\ref{X-u-loc}), and
the directions $\frac{\partial}{\partial y_i}$, for all  $i=1, \dots, n$, are tangent to
$ \Sigma_0(t) $ along $ \theta(t)$, because  
$d_{\theta(t)}H(\frac{\partial}{\partial y_i})
=\omega_{\theta(t)}( X_{H_0},\frac{\partial}{\partial y_i} ) \equiv 0$.
Then  the vector fields  $ Y_{u,t} $   and $X_u|_{\Sigma_0(t)} $ in
$\Sigma_0(t)$ have the same $k$-jet along  $ \theta(t) $.
Moreover,
since  $ \widehat{P}_{0,t}\left(\Sigma_0(0)\right)=\Sigma_0(t) $, we have that the
non autonomous vector field $ \widehat{P}^*_{0,t} ( X_u|_{\Sigma_0(t)} )$  is
$k$-tangent to $ \Sigma_0(0) \subset \Lambda(0)$.
Therefore
the submanifolds $\Sigma_0(0) $ and $R_{u,t}( \Sigma_0(0) )$ have a
tangency  of order $k$ at $\theta(0)$. 
Then, to study the $k$-jet
of $R_{u,t}$ at $\theta(0) $ we can assume that $R_{u,t}$
leaves $ \Sigma_0(0) $ invariant for all $ t \in [0,a]$.
\end{rem}
 
\begin{lem}
\label{J-S_t-sigma(0)}
The $k$-jet at $\theta(0)$ of $R_{u,t}|_{\Sigma_0(0)}$ is equal to 
the $k$-jet at $ \theta(0)$ of the Hamiltonian flow at  time $t$ that corresponds to 
the non autonomous Hamiltonian function  
$[\delta(t) F \circ ( \widehat{P}_{0,t}|_{\Sigma_0(0)} ) ]$ in $ \Sigma_0(0)$.
\end{lem}

\proof
Combining Proposition \ref{O-G} and Remark \ref{Sigma-inv},
we conclude that the $k$-jet of $R_{u,t}|_{\Sigma_0(0)}$ is equal to the
$k$-jet of the flow at time  $ t $ associated to the field $ \widehat{P}_{0,t}^*
\left( Y_{u,t}\right)$. On the other hand, if $ X $ denotes the
Hamiltonian field for the non autonomous Hamiltonian 
$[K_{u,t}\circ\left( \widehat{P}_{0,t}|_{\Sigma(0)}\right)] $, then, 
using that $\widehat{P}_{0,t}|_{\Sigma(0)}: \Sigma(0) \rightarrow \Sigma(t) $ is a
symplectic map, we have:
\begin{align*}
\omega(X,\cdot)|_{\Sigma(0)} 
&=  d \left(K_{u,t}\circ\left(\widehat{P}_{0,t}|_{\Sigma(0)}\right)\right) 
= \widehat{P}_{0,t}^* (d \ K_{u,t}) 
= \widehat{P}_{0,t}^*\omega(\  Y_{u,t}\ ,\  \cdot\  )|_{\Sigma(t)} \\
&= \omega( \widehat{P}_{0,t}^*(\ Y_{u,t})\ ,\  \cdot \ )|_{\Sigma(0)}.
\end{align*}
And, since $ \omega|_{\Sigma(0)}$ is no degenerate, we have that 
$X= \widehat{P}_{0,t}^* (Y_{u,t}) $.
Hence the  $k$-jet of $ \widehat{P}_{0,t}^* (Y_{u,t}) $ at $\theta(0) $ 
is determined by the $(k+1)$-jet of the Hamiltonian 
$[K_{u,t}\circ\left( \widehat{P}_{0,t}|_{\Sigma(0)}\right)] $ at $ \theta(0)$,
that, by (\ref{conta-jet}),  is equal to the $k$-jet of the
Hamiltonian 
$\left[\delta(t) F\circ \left(\widehat{P}_{0,t}|_{\Sigma(0)}\right)\right] $.
This completes the proof.
\qed

\begin{rem}
\label{grupo de lie}
Recall that $ J^k_s(n) $ is a  Lie group with the  group
structure  defined by $j^k f\cdot j^k g =j^k(f\circ g)$.
Let $\mathfrak{ J}^k_s(n) $ be the space  of the $k$-jets  at $ 0$
of symplectic vector fields that are zero at $0$.
We define the  bracket 
$[\cdot,\cdot]^k: 
\mathfrak{J}^k_s(n) \times \mathfrak{J}^k_s(n) \rightarrow \mathfrak{J}^k_s(1)$ 
by
$[j^k X , j^k Y ]^k = -j^k [ X , Y ].$
Since  $ X, Y $ are  zero in the origin, 
$[\cdot,\cdot]^k $ depends only  on the  $k$-jets of $ X $ and $Y$.
Then $ [\cdot,\cdot]^k $ defines a   Lie algebra structure in
$ \frak{J}^k_s(n) $.
Moreover,  
$\frak{J}^k_s(n) $ is the Lie algebra of $ J^k_s(n) $ and 
the  exponential map 
$\exp : \mathfrak{ J}^k_s(n) \rightarrow J^k_s(n) $ verifies 
$\exp( t j^k X ) =j^k \phi_t$, 
where $\phi_t$ is the local flow associated  to $ X $. 
For more details  and  proofs, see \cite[\S IV]{Kolar:1993}.
\end{rem}

\begin{defn}
\label{D-kgeneral} 
A one parameter family  of symplectic maps
$\sigma: [a,b]\rightarrow Sp(n)$
is {\em $k$-general}  if 
there exist times $t_1, \dots, t_d \in [a, b]$ such that 
$( \sigma_{t_1}, \dots , \sigma_{t_d} )$ 
is in $ G_k $.
\end{defn}

Recall the definition of the set $G_k$ in Section \ref{k-general maps}.

\smallskip

Next, we show that it is possible to perturb the $k$-jet of the Poincar\'e map
without changing its $(k-1)$-jet.
To do this, 
let
\[
S_a^k : \mathcal F^k \to \ker(\pi_k) \subset J^k_s(n)
\]
be the map defined by
\[
S_a^k(u)= j^k ( R_{u,a}|_{\Sigma_0(0)} ) (\theta(0))
\]
where $\pi_k : J^{k}_s(n) \rightarrow J^{k-1}_s(n) $ is   the canonical
projection.

\begin{prop}
\label{S_t(F)} 
If the one parameter family of linear symplectic maps
\[
[0,a]\ni  t \mapsto d_{\theta(0)}\widehat{P}_{0,t}|_{\Sigma_0(0)}  \in Sp(n),
\]
is $(k+1)$-general for some $k \geq 2$, 
then  the map $S_a^k$ is  a local submersion in a neighborhood of  $0 \in \mathcal{F}^k$.
\end{prop}

\proof
Let $d=d(2n, k+1)$ be the dimension of the real vector space of homogeneous polynomials 
of degree $k+1$ in the $2n$ variables $ x=(x_1,\dots, x_n) $ and $ y = (y_1,\dots,y_n)$.
Since $ d_{\theta(0)}\widehat{P}_{0,t}|_{\Sigma(0)}$ is a $(k+1)$-general family 
of symplectic linear maps  for $0 \leq t\leq a$, 
there are  $t_1, \dots, t_d \in (0,a) $, such that 
\[ 
\left\{ 
F( d\widehat{P}_{0,t_1} (x, y) ), \dots, F( d\widehat{P}_{0,t_d} (x, y) )  
\right\}
\] 
is a basis  of  $\R_{k+1}[x, y]$,
where $ F(x, y) = x_1^{k+1}\in\R_{k+1}[x, y]$.
For each $ 1\leq i \leq d$ and  $\lambda >0 $ sufficiently small, 
let $\delta_{t_i}^\lambda : \R \rightarrow \R $ be a
$C^\infty $ approximation of the Dirac delta function at the
point  $ t_i $ with support in the interval  $ [ t_i-\lambda, t_i + \lambda] $.
We consider
$u_i \in \m F^k$
defined by
\[
u_i(x_0, x_1, \dots, x_n)= \delta_{t_i}^\lambda(x_0) \beta (x_1, \dots, x_n),
\] 
for $i= 1, \dots, d$.
By  Lemma \ref{J-S_t-sigma(0)} and the  properties of the 
exponential map, as recalled in Remark  \ref{grupo de lie},
we have that
\[
D_{0} S_t^k\cdot ( u_i) 
= \left.\frac{\partial}{\partial s}\right|_{s=0} S^k_t( s\ u_i) 
= \left.\frac{\partial}{\partial s}\right|_{s=0} \mbox{ exp} ( t\ j^k(s\ X_i)),
\] 
where $ X_i $ denotes the  Hamiltonian field  in $ \Sigma(0) $ 
corresponding to the  non autonomous Hamiltonian 
$ [ \delta_{t_i}^\lambda(t) F \circ ( \widehat{P}_{0,t}|_{\Sigma(0)})]$.
Computing  the derivative
with respect to  $t $ in the above equality, we obtain:
\begin{align*}
\frac{d}{dt} ( D_{0} S_t^k\cdot (u_i) ) 
&= \left. \frac{\partial}{\partial s} \right|_{s=0} 
\left( \frac{\partial}{\partial t} \exp( t\ j^k(s \ X_i)) \right) \\
&= \left.\frac{\partial}{\partial s}\right|_{s=0} 
   \left( d_{( t\ j^k (s\ X_i))}  \exp \cdot  j^k(s\ X_i )\right) \\
&= \left.\frac{\partial}{\partial s}\right|_{s=0} j^k(s\ X_i ) 
= j^k X_i .
\end{align*}

Then
\begin{equation}
\label{eq-ds-k} 
D_{0} S_a^k \cdot ( u_i) = \int_0^a \  j^k X_i  \ dt.
\end{equation}

By definition of  $ X_i $ and  (\ref{eq-ds-k}), we have that, 
if $\lambda $ converges to $ 0 $, 
then  $ D_{0} S_a^k\cdot ( u_i)$ converges to 
the $k$-jet at $\theta(0) $ of the Hamiltonian vector field
in $ \Sigma(0) $ corresponding to the autonomous  Hamiltonian 
$H_i = [ F \circ ( \widehat{P}_{0,t_i}|_{\Sigma(0)} ) ]$.
Computing the  $(k+1)$-jet of $H_i$ at $\theta(0) $, we obtain that
\[
j^{k+1} H_i = [ F\circ ( d_{\theta(0)}\widehat{P}_{0,t_i}|_{\Sigma(0)}) ].
\]
Since
$\{ F \circ ( d_{\theta(0)} \widehat{P}_{t_i}|_{\Sigma(0)} ) : 1\leq i \leq d \}$ 
is a basis  for $ \R_{k+1}[x, y]$, 
we have that, for $\lambda  $ sufficiently small, 
$\{ D_{0} S_a^k\cdot ( u_i) : 1 \leq i \leq d \}$ 
is a basis of the Lie  algebra of the Lie subgroup $\ker( \pi_k)$.
Hence the map $S_a^k $ is a local submersion.
\qed

\bigskip

Now, let $k_0\geq 1$ and $Q$, an open invariant subset of $J_s^{k_0}(n)$,
be given as in the hypothesis of Theorem \ref{T-jet}.
If we assume that the derivatives  $ d_{\theta(0)}\widehat{P}_{0,t}|_{\Sigma_0(0)} $ 
are  $k$- general for each $k= 3, \dots, k_0 + 1$,
it follows from the  proposition above 
that, for  each $k= 2, \dots, k_0$, 
the map $S_a^k$  is an open map  in a neighborhood of 
$0\in \mathcal{F}^k \subset \mathcal{F}^1$. 
And this, together with the hypothesis
$j^{k_0} P(\theta,\Sigma,H_0) \in \overline{Q}$, 
implies that
there exists $u \in \mathcal F^1 $ arbitrarily $ C^\infty$-close to zero such that 
the $k_0$-jet of $R_{u,a}|_{\Sigma_0(0)}$ is in $P^{-1} Q$.   
Then, $j^{k_0} P\in Q$.

Then to complete the  proof of Theorem \ref{T-jet}, it only remains to verify that 
it is possible to perform a perturbation that makes the derivatives  
$d_{\theta(0)}\widehat{P}_{0,t}|_{\Sigma_0(0)} $ 
$k$-general, for each $ k = 3, \dots, k_0 + 1$.
This is the content of the next lemma.

\begin{lem}
\label{P-dp-kgen}
For each integer  $ k>2$,  there exists a smooth potential $ u : M \rightarrow\R$,
arbitrarily close to $0$ in the $C^\infty$-topology, such that 
$\theta(t) $ is also a closed orbit of the Hamiltonian flow of $H_u$ and
the one parameter family $t \mapsto d_{\theta(0)} P_{u,t}$, for $t \in[0,a]$,  
is $k$-general. 
\end{lem}

\proof
Consider the map $S^1_t: \m{F}^1 \rightarrow Sp(n)$
defined by 
\[
S^1_{t} (u)= d_{\theta(0)} P_{u,t} .
\]
Given $ k \geq 2 $, let  $d=d(2n, k+1)$ be the dimension of the real vector space 
of homogeneous polynomials of degree $k+1$ in the $2n$ variables 
$ x=(x_1,\dots, x_n) $ and $ y = (y_1,\dots,y_n)$.
Choosing times  $ 0< t_1< \dots < t_d < a$ we define the map
\[
\Phi: \mathcal{F}^1 \rightarrow  Sp(n)^d
\]
given by
\[
\Phi(u)= \left( S^1_{t_1}(u), \dots, S^1_{t_d}(u) \right).
\]
Then,  using  Proposition 4.2 and Lemma 4.3 in \cite{RR:2011},  we obtain that 
each component of the map $\Phi$  is a local submersion near $0\in \mathcal{F}^1$.
Since $G_k \subset Sp(n)^d$ is dense  (Proposition~\ref{P-kgen}),
we can choose  a potential
\[
u=  u_1 +\dots +u_d \in \m{F}^1
\] 
with $\supp(u_i) \cap \gamma([t_{j-1}, t_j])= \emptyset$ for $i\neq j$, 
arbitrarily close to $0$ in the $C^\infty$-topology,
and  such that  the one parameter family of linear symplectic maps associated to 
the linearized Poincar\'e map of the perturbation  $ H_u $ is $k$-general.
\qed

\section{Proof of Theorem \ref{prop-P-Q}}

Let $ Q $ be  an open, dense and invariant subset of $ J_s^k(n) $ and $ c \in \R $.
As in the Kupka-Smale Theorem, for each $ m \in \N$, 
let $ \m{G}(m) \subset C^\infty(M)$ be such that 
$H_u^{-1}(c) $ are  regular and  every closed  orbit of 
the Hamiltonian flow of $H_u$ in $H_u^{-1}(c)$ of period less than or equal to $m$ 
are non-degenerate, for all $ u \in \m {G}(m)$.
Then $ \m{G}(m) $ is an open and dense 
subset of $C^\infty(M)$ with the $  C^\infty$-topology  
(see Lemma 3.3 in \cite{Oliveira:2008}).   
Let $ \m O(m) \subset \m {G}(m)$ be the set of $ C^\infty$-potentials $ u:M\to \R $ 
such that the $k$-jet of the Poincar\'e map of every closed orbit of 
the Hamiltonian flow of $H_u$ in $H_u^{-1}(c)$ of period less than or equal to $m$ 
belongs to $ Q$.
Since the set of  periodic orbits of  
the Hamiltonian flow of $H_u$ in $H_u^{-1}(c)$ of period less than or equal to $m$ 
is finite, for all $ u \in \m O(m)$, and by continuity of the Poincar\'e map, 
we have that $ \m O(m) $ is open.
By Theorem \ref{T-jet}, $\m O(m) $ is also 
a $C^\infty$-dense subset of   $C^\infty(M)$.
Therefore 
\[
\m{O}= \bigcap_{m\in \N} \m{O}(m)
\] 
is the residual subset that we were looking for.
\qed

\section*{Acknowledgments}
We are grateful to G. Contreras and M. J. Dias Carneiro, for the helpful conversations.
We also thank the referee, who kindly pointed out a few mistakes.

\bibliographystyle{amsalpha}


\providecommand{\bysame}{\leavevmode\hbox to3em{\hrulefill}\thinspace}
\providecommand{\MR}{\relax\ifhmode\unskip\space\fi MR }
\providecommand{\MRhref}[2]{\href{http://www.ams.org/mathscinet-getitem?mr=#1}{#2}}
\providecommand{\href}[2]{#2}

\end{document}